\documentclass{amsart}

\usepackage{amssymb}
\usepackage{url} 

\newtheorem{theorem}{Theorem}
\newtheorem{proposition}[theorem]{Proposition}
\newtheorem{lemma}[theorem]{Lemma}
\newtheorem{corollary}[theorem]{Corollary}
\theoremstyle{definition}
\newtheorem{definition}[theorem]{Definition}
\newtheorem{definition and remark}[theorem]{Definition and Remark}
\newtheorem{remark and definition}[theorem]{Remark and Definition}
\newtheorem{example}[theorem]{Example}
\newtheorem{remark}[theorem]{Remark}

\newtheorem{notation and remark}[theorem]{Notation and Remark}

\DeclareMathOperator{\rank}{rank}

\newcommand{\Gm}{{\mathfrak m}}

\newcommand{\length}{\text{\rm length}}

\newcommand{\Soc}{{\rm Soc}}

\newcommand{\Hilb}{{\rm H}}

\renewcommand{\ell}{{l}}

\def\Gin{{\rm Gin}}

\begin{document}






\title[Weak Lefschetz property and ${\mathfrak m}$-full ideals]{The weak Lefschetz property 
for  ${\mathfrak m}$-full ideals and componentwise linear ideals}


%

\author{Tadahito Harima}
\address{Department of Mathematics Education, Ehime University, 
    Matsuyama 790-8577, JAPAN
}

\email{harima@ed.ehime-u.ac.jp}

\author{Junzo Watanabe}
\address{Department of Mathematics, Tokai University, 
    Hiratsuka 259-1292, JAPAN
}

\email{watanabe.junzo@tokai-u.jp}


%

\thanks{The first author was  supported by Grant-in-Aid for Scientific Research (C) 
(23540052)}

%

\begin{abstract}
We give a necessary and sufficient condition 
for a standard graded  Artinian ring of the form $K[x_1, \ldots, x_n]/I$, 
where $I$ is an ${\mathfrak m}$-full ideal, 
to have the weak Lefschetz property in terms of graded Betti numbers. 
This is a generalization of a theorem of Wiebe for componentwise linear ideals. 
We also prove that the class of 
componentwise linear ideals and that of completely ${\mathfrak m}$-full ideals 
coincide in characteristic zero and in positive characteristic, with the assumption 
that $\Gin(I)$ w.r.t.\ the graded reverse lexicographic order is stable. 
\end{abstract}

\maketitle




\section{Introduction}

In \cite{Wiebe} Wiebe gave a necessary and sufficient condition, 
among other things,  
for a componentwise linear ideal 
to have the weak Lefschetz property 
in terms of the graded Betti-numbers.  
His result says that 
the following conditions are equivalent for a  componentwise linear ideal $I$ 
in the polynomial ring  $R$.   
\begin{enumerate}
\item[(i)] 
$R/I$ has the weak Lefschetz property.
\item[(ii)] 
$\beta_{n-1,n-1+j}(I)=\beta _{0,j}(I)$ for all $j>d$. 
\item[(iii)] 
$\beta_{i,i+j}(I)={n-1 \choose i}\beta _{0,j}(I)$ 
for all $j>d$ and all $i\geq 0$. 
\end{enumerate}
Here $d$ is the minimum of all $j$ with $\beta_{n-1, n-1+j} (I) > 0$. 
It seems to be an interesting problem to consider 
if there exist classes of ideals other than componentwise linear ideals 
for which these three conditions are equivalent.    

Recently  Conca, Negri and  Rossi~\cite{CNR} 
obtained a result which says that componentwise linear ideals are $\Gm$-full. 
Bearing this in mind, 
one might expect that for $\Gm$-full ideals 
these three conditions might be  equivalent.   
The outcome is not quite what the authors  had expected; 
if $I$ is  $\Gm$-full, 
then although  
 (i) and (iii)  are equivalent, 
the condition (ii) has to be strengthened slightly. 
This is stated in Theorem~\ref{main theorem} below.  
What is interesting is that 
if we assume that both  $I$ and $I + (x)/(x)$  are $\Gm$-full, 
where $x$ is a general linear form,  then it turns out that 
these three conditions are precisely equivalent.  
This is the first main result in this paper 
and is stated in Theorem~\ref{second_main_thm}.   
As a corollary 
we get that for completely $\Gm$-full ideals 
the conditions (i)-(iii) are equivalent. 

In the above cited paper \cite{CNR} the authors in fact proved, 
without stating it, 
that if  $I$ is componentwise linear, 
it is completely $\Gm$-full. 
We supply a new proof in Proposition~\ref{componentwise}, 
using some results of Conca-Negri-Rossi (\cite{CNR}, Propositions 2.8 and 2.11).   
Hence one sees that 
Theorem~\ref{second_main_thm} is a generalization 
of the theorem of Wiebe  (\cite{Wiebe}, Theorem 3.1).   
Proposition~\ref{componentwise} suggests 
investigating whether all completely $\Gm$-full ideals 
should be componentwise linear. 
In Theorem~\ref{completely_m_full_is_componentwise_liear},  
we prove that it is indeed true with the assumption that 
the generic initial ideal with respect to 
the graded reverse lexicographic order is stable.   
This is the second main theorem in this paper. 
Thus we have a somewhat  striking fact that 
the class of completely $\Gm$-full ideals 
and that of componentwise linear ideals coincide at least in characteristic zero, 
since the additional assumption is automatically satisfied in characteristic zero. 
We have been unable to prove it without the assumption that the generic initial ideal is stable, 
but we conjecture it is true. 
In Example~\ref{ex333} 
we provide a rather trivial example which shows that 
Theorem~\ref{second_main_thm} is not contained in  Wiebe's result.

In Section~2, we give some remarks on 
a minimal generating set of an $\Gm$-full ideal, 
and also review a result on graded Betti numbers obtained in  \cite{Watanabe2}. 
These are needed for our proof of the main theorems.  
In Sections~3 and 4, we will prove the main theorems. 

The authors thank the referee and S. Murai 
for suggesting an improved version of the positive characteristic case 
in Theorem~\ref{completely_m_full_is_componentwise_liear}.

Throughout this paper, 
we let $R=K[x_1,\ldots,x_n]$ be the polynomial ring in $n$ variables 
over an infinite field $K$ 
with the standard grading, 
and $\Gm=(x_1,\ldots,x_n)$ the homogeneous maximal ideal.    
All the ideals we consider are homogeneous.


\section{Some properties of $\Gm$-full ideals}

We quickly review some basic properties of $\Gm$-full ideals, 
which were mostly obtained in \cite{Watanabe2}, 
and fix notation which we use throughout the paper. 
We start with  a  definition.

\begin{definition}[\cite{Watanabe}, Definition 4] 
An ideal $I$ of $R$ is said to be {\em $\Gm$-full} 
if there exists an element  $x$ in $R$ such that 
$\Gm I:x = I$. 
\end{definition}

$\Gm$-full ideals are studied in 
\cite{CNR}, \cite{G}, \cite{GH}, \cite{HLNR}, 
\cite{Watanabe}, \cite{Watanabe2} and \cite{Watanabe3}.

\begin{notation and remark}\label{notation of mfull} 
\begin{enumerate}
\item[(1)] 
Suppose that $I$ is  an $\Gm$-full ideal of $R$.  
Then the equality $\Gm I:x = I$ holds 
for a general linear form $x$ in $R$ (\cite{Watanabe}, Remark 2 (i)). 
Moreover, it is easy to see that, for any $x \in R$, if  
$\Gm I :x = I$, then it   implies that  $I:\Gm = I:x$. 
Let $y_1,\ldots,y_\ell$ be homogeneous elements in $I:\Gm$ 
such that $\{\overline{y_1},\ldots,\overline{y_\ell}\}$ is 
a minimal generating set for $(I:\Gm)/I$, 
where $\overline{y_i}$ is the image of $y_i$ in $R/I$. 
Then Proposition 2.2 in \cite{GH} implies that 
$\{xy_1,\ldots,xy_\ell\}$ is a part of a minimal generating set of $I$. 
Write a minimal generating set of $I$ as 
\begin{align*}
xy_1,\ldots,xy_\ell, z_1,\ldots,z_m.
\end{align*}

\item[(2)] 
Suppose that 
$I$ is an $\Gm$-primary $\Gm$-full ideal of $R$. 
The socle of $R/I$ is the ideal of $R/I$ annihilated 
by the maximal ideal $\overline{\Gm}=(\overline{x_1},\ldots,\overline{x_n})$.  
Hence 
\begin{align*}
\Soc (R/I) = \{a\in R/I \mid a\overline{\Gm} = 0\} 
= \oplus_j \Soc (R/I)_j \cong (I:\Gm)/ I. 
\end{align*}
The socle degree of $R/I$ is the maximum of integer $j$ 
with $\Soc(R/I)_j \neq 0$. 
We note that 
\begin{align*}
\max\{\deg (xy_i) \mid i=1,\ldots,\ell\} = c+1, 
\end{align*}
where $c$ is the socle degree of $R/I$. 
\item[(3)] 
Let $x$ be a general linear form in  $R$.  
Let $\delta$ be the minimum integer $j$ with $(R/(I+xR))_j=0$ 
and let $d$ be the least of integers $\{\deg (xy_i) \mid i=1,\ldots,\ell\}$. 
Then $\delta-1$ is equal to the socle degree of $R/(I+xR)$ 
and $d-1$ is equal to the initial degree of $\Soc (R/I)$, 
i.e., $d-1=\min\{j \mid \Soc(R/I)_j\neq 0\}$.  
Note that $\delta$ is independent of $x$ provided  that it is sufficiently  
general.  This is discussed in  Remark~\ref{remark on delta} below.  
\item[(4)] 
Let $\beta_{i,j} (I)$ be the $(i,j)$th graded Betti number of $I$ 
as an $R$-module. 
Then we have $d=\min\{ j \mid \beta_{n-1, n-1+j} (I) > 0 \}$, 
as  $\dim_K\Soc(R/I)_j=\beta_{n-1,n+j}(I)$ (\cite{Wiebe}, Fact 3.3).  
\end{enumerate}
\end{notation and remark}

\begin{remark}\label{remark on delta}
Let  $\delta$  be the integer defined in Remark~\ref{notation of mfull} (3). 
In this remark we want to prove that 
$\delta$ is independent of $x$ as long as  it is sufficiently general.  

Let $I$ be an $\Gm$-primary ideal in  $R=K[x_1,\ldots,x_n]$ 
and set $A=R/I$. 
Let $\xi _1,\ldots, \xi _n$ be indeterminates over $K$. 
Let $K(\xi)=K(\xi _1, \ldots, \xi _n)$ be the rational function field over $K$ and 
put   $A(\xi)=K(\xi) \otimes _K A$.  
It is easy to see that both $A(\xi)$ and $A$ have the same Hilbert function. 
Put $Y=\xi _1\overline{x_1}+\cdots + \xi _n\overline{x_n}\in A(\xi)$.   
It is  proved, in a more general setup  in \cite[Theorem A]{Watanabe}, that  
\begin{align}
\length(A(\xi)/YA(\xi)) \leq \length(A/yA) 
\end{align}
for any linear form $y$ of $A$ and 
\begin{align} \label{inequality1 for delta}
\length(A(\xi)/YA(\xi))  = \length(A/yA)     
\end{align}
for any  sufficiently general linear form $y$ of $A$ 

We  show that 
$
\Hilb(A(\xi)/YA(\xi),i) \leq \Hilb(A/yA,i)
$
for every $i$, 
where
$\Hilb(\ast, i)$ is the Hilbert function of a graded algebra. 
Let $A(\xi)_i$ be the homogeneous part of  $A(\xi)$  of   degree $i$. 
Choose a homogeneous basis $\{v_{\lambda}\}$ for $A$ as a vector space over $K$.  
Then $\{1 \otimes v _{\lambda}\}$ 
is a homogeneous  basis for  $A(\xi)$ over $K(\xi)$.   We fix any pair of such  bases and 
suppose that we write the homomorphisms 
$\times Y: A(\xi)_i \rightarrow A(\xi) _{i+1}$ and  
$\times y: A(\xi)_i \rightarrow A(\xi) _{i+1}$ as matrices  $M_i$ and $N_i$ respectively over these bases.   
It is easy to see that the entries of $M_i$ are  homogeneous linear forms  in $\xi_1, \ldots, \xi_n$, and  
$N_i$ is obtained from $M_i$ by substituting $(\xi _j)$ for $(a _j)$, where $y=a_1x_1 + \cdots + a_nx_n$ with $a_i \in A$.     
Thus  $\rank M_i \geq \rank N_i$ 
and consequently   
\begin{equation}\label{inequality2 for delta}
\Hilb(A(\xi)/YA(\xi),i)\leq \Hilb(A(\xi)/yA(\xi),i)=\Hilb(A/yA,i)
\end{equation}
for every  $i$ and every linear form $y \in A$. 
By (\ref{inequality1 for delta}) and (\ref{inequality2 for delta}) we see  that 
both $A(\xi)/YA(\xi)$ and $A/y A$ have the same Hilbert function, 
if $y \in A$ is a sufficiently general linear form. 
This shows that  $\delta$ is independent of a choice of sufficiently  general linear form $y \in A$.   
\end{remark}

\begin{lemma}\label{gen of mfull} 
Let $I$ be an  $\Gm$-primary  $\Gm$-full ideal of $R$.  
Let $x$,  $z_1, \cdots, z_m$ and $\delta$  be as in Notation~\ref{notation of mfull}. 
Let $\overline{z_i}$ be the image of $z_i$ in $R/xR$ 
and $\overline{I}$ the image of $I$ in $R/xR$.  Then we have: 
\begin{enumerate}
\item[$(1)$] 
$\{\overline{z_1},\ldots,\overline{z_m}\}$ is a minimal generating set of $\overline{I}$. 
\item[$(2)$]
$\deg z_i \leq \delta$ for  $i=1, 2, \ldots, m$. 
\item[$(3)$] 
If $\overline{I}$ is an  $\Gm$-full ideal  of $R/xR$, 
then $\deg z_i = \delta$ for some $i$. 
\end{enumerate}
\end{lemma}

\begin{proof} 
(1) Suppose that $z_1\in (z_2,\ldots,z_m,x)$. 
Then $z_1=f_2z_2+\cdots+f_mz_m+f_{m+1}x$ for some $f_i\in R$. 
Since $xf_{m+1}=z_1-(f_2z_2+\cdots+f_mz_m)\in I$, 
we have 
\begin{align*}
f_{m+1} \in I:x = I: \Gm=(y_1,\ldots,y_{\ell}, z_1,\ldots,z_m). 
\end{align*}
Hence 
\begin{align*}
f_{m+1}=g_1y_1+\cdots+g_{\ell}y_{\ell}+h_1z_1+\cdots+h_mz_m 
\end{align*}
for some $g_i, h_j\in R$. 
Thus we obtain 
\begin{align*}
z_1-xh_1z_1=(f_2+xh_2)z_2+\cdots+(f_m+xh_m)z_m+g_1xy_1+\cdots+g_{\ell}xy_{\ell},  
\end{align*}
and $z_1\in (xy_1,\ldots,xy_{\ell},z_2,\ldots,z_m)$. 
This is a contradiction. 

(2) 
Since  $\delta -1$ is equal to the socle degree of $R/(I+xR)$, 
(1) implies that 
\begin{align*}
\deg z_i = \deg \overline {z_i} \leq \delta.  
\end{align*} 

(3) 
This is proved by  applying Remark~\ref{notation of mfull} (2) 
to the $\Gm$-full ideal $\overline{I}$ of $R/xR$. 
\end{proof}

\begin{proposition}[\cite{Watanabe2}, Corollary 8] \label{Betti numbers of mfull} 
Let $I$ be an $\Gm$-full ideal of $R$ (not necessarily $\Gm$-primary) and let 
$x$ be a general linear form of  $R$ satisfying $\Gm I : x=I$. 
With Notation~\ref{notation of mfull} (1), 
let $\overline{I}$ be the image of $I$ in $R/xR$ 
and let $\beta_{i,j} (\overline{I})$ be 
the $(i,j)$th graded Betti number of $\overline{I}$ as an $R/xR$-module. 
Set $c_j=\#\{i \mid 1\leq i \leq \ell, \deg (xy_i)=j\}$  
for all $j$. 
Then 
\begin{align*}
\beta_{i,i+j} (I) = \beta_{i,i+j} (\overline{I}) + {n-1\choose i} c_j 
\end{align*} 
for all $i$ and all $j$. 
\end{proposition}


\section{$\Gm$-Full ideals and the WLP}

\begin{definition} \label{weak Lefschetz property} 
Let $I$ be an $\Gm$-primary ideal of $R$ 
and set $A=R/I=\oplus_{i=0}^c A_i$. 
We say that $A$ has the {\em weak  Lefschetz property} (WLP)  
if there exists a linear form $L \in A_1$ 
such that the multiplication map $\times L:A_i \to A_{i+1}$ 
has full rank for all $0\leq i\leq c-1$. 
\end{definition}

\begin{remark} \label{remark of WLP} 
Let $I$ be an $\Gm$-primary ideal of $R$ 
and set $A=R/I=\oplus_{i=0}^c A_i$. 
\begin{enumerate} 
\item[(1)] 
Let $h_0,h_1,\ldots,h_c$ be the Hilbert function of $A$. 
Then it is easy to see that 
the following conditions are equivalent. 
\begin{enumerate} 
\item[(i)] 
$A$ has the WLP. 
\item[(ii)]
There exists a linear form $L \in A_1$ such that 
\begin{align*}
\dim_K\ker (\times L:A_i \to A_{i+1}) = \max\{0, h_i-h_{i+1}\} 
\end{align*} 
for all $0\leq i\leq c-1$. 
\end{enumerate}
\item[(2)]
If $A$ has the WLP, 
then for a general linear form $L$ in $A$, 
the multiplication map $\times L :A_i \to A_{i+1}$ 
has full rank for all $0\leq i\leq c-1$. 
\item[(3)] 
When we deal with the WLP of Artinian algebras defined by $\Gm$-full ideals, 
we need to select a linear form  with both properties  in the definition of $\Gm$-full ideals  and 
in the definition of WLP.  This is possible since either property is satisfied by a 
sufficiently general linear form.    
\end{enumerate}
\end{remark}

\begin{lemma}\label{eq1 of WLP}
Let $I$ be an $\Gm$-primary $\Gm$-full ideal of $R$. 
Then,  with $\delta$ and $d$ as defined in Notation~\ref{notation of mfull}, 
 the following conditions are equivalent. 
\begin{enumerate}
\item[$(i)$] 
$R/I$ has the WLP. 
\item[$(ii)$]
$\delta \leq d$. 
\end{enumerate}
\end{lemma}

\begin{proof} 
Let $x$ be a general linear form in $R$. 
Note that $\delta$ is equal to the minimum of integers $j$ 
such that 
the multiplication map $\times\overline{x}: (R/I)_{j-1}\rightarrow (R/I)_j$ 
is surjective. 
Furthermore note that 
$\times\overline{x}: (R/I)_{j-1}\rightarrow (R/I)_j$ is injective for all $j\leq d-1$, 
since $\Soc (R/I)=\ker (\times\overline{x}: (R/I)\rightarrow (R/I))$ 
and $d-1$ is equal to the initial degree of $\Soc(R/I)$. 
Hence $(ii)$ $\Rightarrow$ $(i)$ as is easily seen. 
Assume that $ d < \delta $. 
Then the map  
$\times\overline{x}: (R/I)_{d-1}\rightarrow (R/I)_d$ 
is neither surjective nor injective. 
Hence $R/I$ does not have the WLP. 
This shows $(i)$ $\Rightarrow$ $(ii)$. 
\end{proof}

\begin{lemma}\label{eq2 of WLP} 
Let $I$ be an $\Gm$-full ideal of $R$ and $\overline{I}$ the image of $I$ in $R/xR$ 
for a general linear form $x$ in $R$.  Let $\beta _{ij}$ be as in  Notation~\ref{notation of mfull}.
Then the following conditions are equivalent. 
\begin{enumerate}
\item[$(i)$] 
$R/I$ has the WLP. 
\item[$(ii)$]
$\beta_{n-2,n-2+j} (\overline{I}) = 0$ for all $j>d$. 
\end{enumerate}
\end{lemma}

\begin{proof} 
Recall that  $d$ is the  minimum  $j$ such that  
$\beta_{n-1,n-1+j} (I) > 0$.  (See Notation~\ref{notation of mfull} (3) and (4).) 
Since 
\begin{align*}
\dim_K \Soc(\overline{R}/\overline{I})_j = \beta_{n-2,n-1+j} (\overline{I}) 
\end{align*} 
for all $j$, 
it follows from Remark~\ref{notation of mfull} (3) that 
\begin{align*}
\delta-1=\max\{ j \mid \beta_{n-2,n-1+j} (\overline{I})>0 \}. 
\end{align*} 
Hence we have the equivalence: 
\begin{align*}
\delta \leq d \Leftrightarrow \beta_{n-2,n-2+j} (\overline{I})=0 
\ \mbox{for all $j>d$}. 
\end{align*}
Thus the assertion follows from Lemma~\ref{eq1 of WLP}.  
\end{proof}

Now we state the first theorem. 

\begin{theorem}\label{main theorem} 
Let $R=K[x_1,\ldots, x_n]$ be the polynomial ring 
in $n$ variables over an infinite field $K$ 
and $\Gm=(x_1,\ldots,x_n)$ the homogeneous maximal ideal. 
Let $I$ be an $\Gm$-primary $\Gm$-full ideal of $R$ 
and let $d$ be the minimum of all $j$ with $\beta_{n-1, n-1+j} (I) > 0$. 
Then the  following conditions are equivalent.  
\begin{enumerate}
\item[$(i)$] 
$R/I$ has the WLP. 
\item[$(ii)$] 
$\beta_{n-1, n-1+j} (I) =\beta_{0, j} (I)$ 
and $\beta_{n-2, n-2+j} (I)=(n-1)\beta_{0,j} (I)$,  
for all $j>d$. 
\item[$(iii)$] 
$\beta_{i, i+j} (I) ={n-1 \choose i}\beta_{0, j} (I)$  
for all $j>d$ and all $i$. 
\end{enumerate}
\end{theorem}

\begin{proof} 
We use Notation~\ref{notation of mfull}. 
Let $\overline{I}$ be the image of $I$ in $R/xR$.  
Noting that  
$\beta_{n-1,n-1+j} (\overline{I})=0$ for all $j$, 
we have from Proposition~\ref{Betti numbers of mfull} that   
\begin{align}
c_j=\beta_{0,j}(I)  \ \mbox{for all $j>d$} \Leftrightarrow 
\beta_{n-1,n-1+j}(I)=\beta_{0,j}(I) \ \mbox{for all $j>d$}. 
\end{align} 

(i) $\Rightarrow$ (ii): 
By Lemmas~\ref{eq1 of WLP} and~\ref{gen of mfull} (2), 
it follows that 
\begin{align}
c_j=\beta_{0,j}(I) 
\end{align} 
for all $j>d$. 
Hence we have the first equality by the equivalence (1). 
Moreover, by Lemma~\ref{eq2 of WLP} and Proposition~\ref{Betti numbers of mfull}, 
it follows that $\beta_{n-2,n-2+j}(I)=(n-1)c_j$ for all $j>d$. 
Hence we have the second equality by the above equality (2). 

(ii) $\Rightarrow$ (i): 
By our assumption (ii), 
Proposition~\ref{Betti numbers of mfull} and the equivalence (1), 
we have $\beta_{n-2,n-2+j}(\overline{I})=0$ for all $j>d$. 
Hence (ii) $\Rightarrow$ (i) follows by Lemma~\ref{eq2 of WLP}. 

(ii) $\Rightarrow$ (iii): 
Since $\beta_{n-2,n-2+j}(\overline{I})=0$ for all $j>d$, 
we have $\beta_{i,i+j}(\overline{I})=0$ for all $i$ and all $j>d$. 
Hence, noting that $\beta_{0,j}(I)=c_j$ for all $j>d$, 
we see that the desired equalities follow 
from Proposition~\ref{Betti numbers of mfull}. 

(iii) $\Rightarrow$ (ii) is trivial. 
\end{proof}

The above theorem can be strengthened as follows.

\begin{theorem} \label{corollary of main theorem}
\label{second_main_thm} 
With the same notation as Theorem~\ref{main theorem}, 
suppose that $\Gm I:x=I$ and 
$\overline{I}=(I+xR)/xR$ is $\Gm$-full as an ideal of $R/xR$ 
for some linear form $x$ in $R$. 
Then the  following conditions are equivalent.  
\begin{enumerate}
\item[$(i)$] 
$R/I$ has the WLP. 
\item[$(ii)$] 
$\beta_{n-1, n-1+j} (I) =\beta_{0, j} (I)$ for all $j>d$. 
\item[$(iii)$] 
$\beta_{i, i+j} (I) ={n-1 \choose i}\beta_{0, j} (I)$  
for all $j>d$ and all $i$. 
\end{enumerate}
\end{theorem}

\begin{proof} 
In view of Theorem~\ref{main theorem}, 
it suffices to show the assertion (ii) $\Rightarrow$ (i). 
We use Notation~\ref{notation of mfull}.  
Recall that $d=\min\{\deg(xy_s) \mid 1\leq s \leq \ell\}$, 
as explained in Remark~\ref{notation of mfull} (3) and (4). 
By Proposition~\ref{Betti numbers of mfull}, 
we have $\beta_{n-1, n-1+j}(I)=c_j$ for all $j$, 
since $\beta_{n-1,n-1+j}(\overline{I})=0$ for all $j$. 
Now assume (ii). 
Then $\beta_{0,j}(I)=c_j$ for all $j>d$. 
Therefore, we have 
$\deg(z_t) \leq d$ for all $t=1,2,\ldots,m$. 
On the other hand, 
since we assume that $\overline{I}$ is $\Gm$-full, 
we have $\max\{\deg(z_t) \mid 1\leq t \leq m\}=\delta$ 
by Lemma~\ref{gen of mfull} (3). 
Thus $\delta \leq d$, 
and $R/I$ has the WLP by Lemma~\ref{eq1 of WLP}.    
\end{proof}

\begin{corollary} \label{second theorem} 
With the same notation as Theorem~\ref{main theorem},  
let $I$ be an $\Gm$-primary completely $\Gm$-full ideal of $R$ 
(see Definition~\ref{completely m-full} in the next section).  
Then the  following conditions are equivalent.  
\begin{enumerate}
\item[$(i)$] 
$R/I$ has the WLP. 
\item[$(ii)$] 
$\beta_{n-1, n-1+j} (I) =\beta_{0, j} (I)$ for all $j>d$. 
\item[$(iii)$] 
$\beta_{i, i+j} (I) ={n-1 \choose i}\beta_{0, j} (I)$  
for all $j>d$ and for all $i$. 
\end{enumerate}
\end{corollary}

\begin{proof} 
This follows from Theorem~\ref{corollary of main theorem}. 
\end{proof}

Proposition~\ref{componentwise} in the next section says that 
Theorem~\ref{corollary of main theorem} is a generalization of Theorem 3.1 
in \cite{Wiebe} due to Wiebe. 
Furthermore, the following example shows 
that Theorem~\ref{corollary of main theorem} 
is not contained in Wiebe's result.

\begin{example}  \label{ex333}   
Let $R=K[w,x,y,z]$  be the polynomial ring in four variables. 
Let $I=(w^3, x^3, x^2y) + (w,x,y,z)^4$. 
Then, it is easy to see that 
$I$ is not componentwise linear, but 
$I$ and $I+(z)/(z)$ are $\Gm$-full ideals.  
\end{example}

Finally in  this section, 
we give an example of $\Gm$-full ideal 
where conditions (ii) and (iii) of Theorem~\ref{corollary of main theorem} 
are not equivalent. 
In other words, 
the condition ``$I+gR/R$ is $\Gm$-full'' 
cannot be dropped in this theorem.

\begin{example}  \label{new_ex_showing_ii_and_iii_are_not_equiv}   
Let $R=K[x,y,z]$  be the polynomial ring in three variables. 
Let $I=(x^3, x^2y, x^2z, y^3) + (x,y,z)^4$. 
Then, it is easy to see that 
$I$ is $\Gm$-full and $I+(g)/(g)$ is not $\Gm$-full, 
where $g$ is a general linear form of  $R$. 
The Hilbert function of $R/I$ is  $1+3t+6t^2+6t^3$ and that of $R/I+gR$ is $1+2t+3t^2+t^3$.  
Thus $R/I$ is does not have the WLP. 
The minimal free resolution of $I$ is:
\[
0 \rightarrow R(-5)\oplus R(-6)^6 \rightarrow R(-4)^3 \oplus R(-5)^{13} 
\rightarrow R(-3)^4 \oplus R(-4)^6 \rightarrow I \rightarrow 0. 
\]
Note that $d=3$, $\beta _{2,6} =  \beta _{0, 4}$ and $\beta _{1,5}  > 2\beta _{0, 4}$.  
\end{example}


\section{Complete $\Gm$-fullness  and  componentwise linearity}

\begin{definition} [\cite{Watanabe2}, Definition 2] \label{completely m-full}
Let $R=K[x_1,\ldots,x_n]$ be the polynomial ring in $n$ variables 
over an infinite filed $K$, 
and $I$ a homogeneous ideal of $R$. 
We define the {\em completely $\Gm$-full ideals} recursively as follows. 
\begin{enumerate}
\item[(1)] 
If $n=0$ (i.e., if $R$ is a field), then the zero ideal is completely $\Gm$-full.  
\item[(2)] 
If $n>0$, 
then $I$ is completely $\Gm$-full 
if $\Gm I:x=I$ and $(I+xR)/xR$ is completely $\Gm$-full as an ideal of $R/xR$, 
where $x$ is a general linear form. 
(The definition makes sense by induction on $n$.) 
\end{enumerate}
\end{definition}

\begin{definition}
A monomial ideal $I$ of $R=K[x_1,\ldots,x_n]$ is said to be {\em stable} 
if $I$ satisfies the following condition: 
for each monomial $u\in I$, 
the monomial $x_iu/x_{m(u)}$ belongs to $I$ for every $i<m(u)$, 
where $m(u)$ is the largest index $j$ such that $x_j$ divides $u$. 
\end{definition}

\begin{example}\label{stable ideals} 
Typical examples of completely $\Gm$-full ideals are stable monomial ideals. 
Let $I$ be a stable monomial ideal of $R=K[x_1,\ldots,x_n]$. 
First we first  that $\Gm I: x_n=I$. 
Let $w\in\Gm I: x_n$ be a monomial. 
Since $x_nw\in\Gm I$, 
we have $x_nw=x_iu$ for some $x_i$ and a monomial $u\in I$. 
Hence $w=x_iu/x_n\in I$, as $I$ is stable. 
Therefore $\Gm I:x_n\subset I$, 
and other inclusion is clear. 
Furthermore, 
since $\overline{I}=(I+x_nR)/x_nR$ is stable in $\overline{R}=R/x_nR$ again, 
our assertion follows by inductive argument. 
\end{example}

\begin{proposition} \label{componentwise} 
Every componentwise linear ideal of $R=K[x_1,\ldots,x_n]$
is a completely $\Gm$-full ideal. 
\end{proposition}

\begin{proof} 
Let $I=\oplus_{j\geq 0}I_j$ be a componentwise linear ideal 
of $R=\oplus_{j\geq 0}R_j$ 
and let $I_{<d>}=\oplus_{j\geq d}(I_{<d>})_j$ be 
the ideal generated by all homogeneous polynomials of degree $d$ belonging to $I$. 
Then it follows by Lemma 8.2.10 in \cite{HH} that 
$I_{<d>}$ is a componentwise linear ideal for all $d$. 
Recall the result of Conca-Negri-Rossi  \cite{CNR} that 
every componentwise linear ideal is $\Gm$-full. 
Hence $I$ and $I_{<d>}$ are $\Gm$-full ideals for all $d$. 

First we  show the following: 
there exists a common linear form $x$ in $R$ 
such that 
$\Gm I:x=I$ and $\Gm I_{<d>}:x=I_{<d>}$ for all $d$. 
Note that there is a positive integer $k$ such that 
$R_iI_k=I_{k+i}$ for all $i\geq 0$, 
where $R_iI_k=
\{\sum_{\lambda}r_\lambda v_\lambda \mid r_\lambda \in R_i, v_\lambda \in I_k\}$. 
Write $I_{<k>}=\oplus_{j\geq k}(I_{<k>})_j$. 
Then we have $I_{<k+i>}=\oplus_{j\geq k+i}(I_{<k>})_j$ for all $i\geq 0$. 
Hence it is easy to see that 
if $\Gm I_{<k>}:z=I_{<k>}$ for some linear form $z$ 
then $\Gm I_{<k+i>}:z=I_{<k+i>}$ for all $i\geq 0$, 
as $(I_{<k+i>})_j=(I_{<k>})_j$ for all $j\geq k+i$. 
Therefore if $x$ is a sufficiently general linear form, 
we have both $\Gm I:x=I$ and $\Gm I_{<d>}:x=I_{<d>}$ for all $d$. 

To prove this proposition, 
it suffices to show that 
$\overline{I}=(I+xR)/xR$ is also a componentwise linear ideal of $R/xR$, 
because if so, then $\overline{I}$ is $\Gm$-full 
by the above result stated in \cite{CNR}. 
Therefore our assertion follows by inductive argument. 
Let $\overline{I}_{<d>}=\oplus_{j\geq d}(\overline{I}_{<d>})_j$ be 
the ideal generated by all homogeneous polynomials of degree $d$ 
belonging to $\overline{I}$. 
We have to show that $\overline{I}_{<d>}$ has a linear resolution for all $d$. 
We use the same notation as Proposition~\ref{Betti numbers of mfull} for $I_{<d>}$. 
Since $\overline{I}_{<d>}=(I_{<d>}+xR)/xR$, 
it follows by Proposition~\ref{Betti numbers of mfull} that 
\begin{align*}
\beta_{i,i+j} (I_{<d>}) = \beta_{i,i+j} (\overline{I}_{<d>}) + {n-1\choose i} c_j 
\end{align*} 
for all $i$ and all $j$. 
Hence $\overline{I}_{<d>}$ has a linear resolution, 
as $I_{<d>}$ does. 
This completes the proof. 
\end{proof}

From the preceding proof, 
we obtain an immediate consequence. 

\begin{proposition} 
Let $I$ be a componentwise linear ideal of $R=K[x_1,\ldots,x_n]$. 
Then $\overline{I}=(I+xR)/xR$ is  
a componentwise linear ideal of $\overline{R}=R/xR$ 
for a general linear form $x$ in $R$. 
\end{proposition}

We conjecture that a completely $\Gm$-full ideal is componentwise linear.
We have already proved that a componentwise linear ideal is completely $\Gm$-full. 
We prove the converse with the assumption that the generic initial ideal is stable.

\begin{theorem}\label{completely_m_full_is_componentwise_liear}  
Let $I$ be a homogeneous ideal of $R=K[x_1, \ldots, x_n]$ and 
$\Gin(I)$ the generic initial ideal of $I$ 
with respect to  the graded reverse lexicographic order 
induced by $x_1>\cdots>x_n$. 
Assume that $\Gin(I)$ is stable. 
Then $I$ is completely $\Gm$-full 
if and only if $I$ is componentwise linear. 
\end{theorem}

\begin{proof} 
The `if' part is proved in Proposition~\ref{componentwise}. 
So we show the `only if' part. 
Set $J=\Gin(I)$. 
Since $J$ is stable, 
it suffices to show that 
$\beta_{0} (I) = \beta_{0} (J)$ by Theorem 2.5 in \cite{NT}, that is, 
the minimal number of generator of $I$ 
coincides with that of $J$. 
We use induction on the number $n$ of variables. 
The case where $n=1$ is obvious. 
Let $n\geq 2$. 
Since the minimal number of generators of $(I:\Gm)/I$ is equal to 
the dimension of $(I:\Gm)/I$ as a $K$-vector space, 
it follows 
by Proposition~\ref{Betti numbers of mfull} and Remark~\ref{notation of mfull} (1) 
that  
\begin{align*}
\beta_{0} (I) = \beta_{0} (\overline{I}) + \dim_K ((I:\Gm)/I), 
\end{align*} 
where set $\overline{I}=(I+xR)/xR$ for a general linear form $x$ in $R$.  
Similarly, since $\Gm J:x_n=J$, 
$J$ is stable and consequently completely $\Gm$-full 
by Example~\ref{stable ideals}, 
we have 
\begin{align*}
\beta_{0} (J) = \beta_{0} (\overline{J}) + \dim_K ((J:\Gm)/J), 
\end{align*} 
where $\overline{J}=(J+x_nR)/x_nR$. 
First we  show that 
\begin{align*} 
\dim_K ((I:\Gm)/I) = \dim_K ((J:\Gm)/J). 
\end{align*} 
From the exact sequence 
\begin{align*}
0 \rightarrow (I:x)/I \rightarrow R/I \stackrel{\times x}{\rightarrow} 
(I+xR)/I \rightarrow 0, 
\end{align*} 
it follows that 
\begin{align*}
\dim_K ((I:x)/I)_j = \dim_K (R/I)_j - \dim_K (R/I)_{j+1}+\dim_K (R/(I+xR))_{j+1}
\end{align*} 
for all $j$. 
Similarly, we get 
\begin{eqnarray*}
\lefteqn{\dim_K ((J : x_n)/J)_j   = } \\  & &  \dim_K (R/J)_j - \dim_K (R/J)_{j+1} +   \dim_K (R/(J+x_nR))_{j+1}
\end{eqnarray*}

for all $j$. 
Here we recall the well-known facts: 
\begin{itemize}
\item 
$\dim_K (R/I)_j = \dim_K (R/J)_j$ for all $j$. 
\item 
$\dim_K (R/(I+xR))_j = \dim_K (R/(J+x_nR))_j$ for all $j$ (Lemma 1.2 in \cite{C1}). 
\end{itemize}
Therefore we have 
\begin{align*}
\dim_K ((I:x)/I)_j = \dim_K ((J:x_n)/J)_j 
\end{align*}
for all $j$. 
Thus, since $I$ and $J$ are $\Gm$-full, 
it follows that $I:\Gm=I:x$ and $J:\Gm=J:x_n$, 
and hence we see that 
$\dim_K ((I:\Gm)/I) = \dim_K ((J:\Gm)/J)$. 
Furthermore, 
since both $\overline{I}$ and $\overline{J}$ are completely $\Gm$-full again 
and $\Gin(I+xR) = \Gin(I) + x_nR$ (Corollary 2.15 in \cite{Gr}),  
it follows by the inductive assumption that 
\begin{align*}
\beta_{0} (\overline{I}) = \beta_{0} (\overline{J}). 
\end{align*} 
This completes the proof. 
\end{proof}



%

\end{document}